\documentclass[12pt]{amsart}
\usepackage[dvips]{graphicx}
\usepackage{amssymb}
\usepackage{amsmath}
\usepackage{amsthm}
\usepackage{latexsym}
\usepackage[english]{babel}
\usepackage{appendix}
\usepackage{hyperref}
\usepackage{amsmath,amsfonts,latexsym,graphicx,amssymb,url}
\usepackage{amsmath,txfonts,pifont,bbding,pxfonts,manfnt}
\usepackage{psfrag}
\usepackage{wrapfig, subfigure,graphicx}
\usepackage{hyperref}
\usepackage[active]{srcltx}
\usepackage{pdfsync}
\usepackage[none]{hyphenat}

\newtheorem{theorem}{Theorem}[section]
\newtheorem{definition}{Definition}[section]
\newtheorem{lemma}{Lemma}[section]
\newtheorem{proposition}{Proposition}[section]
\newtheorem{remark}{Remark}[section]

\newcommand{\C}{{\mathbb C}}
\newcommand{\R}{{\mathbb R}} 
\newcommand{\N}{{\mathbb N}}

\newcommand{\la}{{\lambda}}
\newcommand{\p}{{\partial}}

\newcommand{\s}{{\sigma}}
\newcommand{\ddb}{{\partial \bar \partial}}

\newcommand{\trace}{\text{trace}}

\begin{document}


\title[Nonsmooth viscosity solutions of complex Hessian PDEs]
{Nonsmooth viscosity solutions of elementary\\
 symmetric functions of the complex Hessian}

\author[Chiara Guidi, Vittorio Martino  \& Annamaria Montanari ]{Chiara Guidi, Vittorio Martino \& Annamaria Montanari}
\address{Dipartimento di Matematica, Universit\`a di Bologna,
piazza di Porta S.Donato 5, 40127 Bologna, Italy.}
\email{chiara.guidi9@studio.unibo.it,
vittorio.martino3@unibo.it,
\newline
annamaria.montanari@unibo.it}

\vspace{5mm}
\begin{abstract}

In this paper we prove the existence of nonsmooth
viscosity solutions for Dirichlet problems involving 
elementary symmetric functions of the
eigenvalues of the complex Hessian.
\vspace{5mm}


\vspace{1mm}
   Elementary symmetric functions of the
eigenvalues of the Complex Hessian. Comparison
principle. Gradient estimates. Pogorelov counterexample
\end{abstract}
\maketitle


\section{Introduction}
\noindent
Let $2\leq k\leq n$. Here we will study the following equation
\begin{equation}\label{sigmak}
\big(\sigma_k(\partial \bar\partial u)\big)^{\frac{1}{k}} = f(z,u,D u)
\end{equation}
where, $\partial \bar\partial u$ is the complex Hessian of $u$ in $\C^n,$ $Du$ is the Euclidean gradient of $u$
and
for every Hermitian $n\times n$ matrix $M$, $\sigma_k(M)$ denotes the $k$-th elementary symmetric function of the eigenvalues of $M$ and $f$ is a positive function.\\

The complex Hessian equation in domains of $\C^n$ was considered by S.-Y. Li  in
\cite{[2004]},
where the author proves existence and uniqueness theorems for the Dirichlet problem for
elliptic nonlinear partial differential equations that are concave symmetric functions of
the eigenvalues of the complex Hessian and where the right hand side $f$ only depends on $z.$ The results are generalizations of those first
established by L. A. Caffarelli, L. Nirenberg and J. Spruck \cite{CNS} and later extended by several authors.

The complex Monge-Amp\`ere equations have been investigated extensively over last
years: we refer the reader to  \cite{K1998}, \cite{P2012} ,\cite{N2014}, \cite{D2014},  and references therein.

In this paper we will show that for any given smooth function $f$, under some suitable structural assumptions (see later), there always exist a small  Euclidean ball $B_r$ and a Lipschitz viscosity solution of (\ref{sigmak}) in $B_r$, which is not of class $C^1(B_r)$ if $k=2$, and it is not of class $C^{1,\beta}(B_r)$, with $\beta>1-\frac{2}{k}$, if $k>2$. 
\\

\noindent
Let us fix some notations: we identify $\mathbb{C}^{n}\approx\mathbb{R}^{2n}$, with $z=(z_1,\ldots,z_{n}), z_j=x_j+iy_j\simeq (x_j,y_j)$, for $j=1,\ldots,n$ and we set
$$u_{z_j}=\frac{\p u}{\p z_j}=\frac{1}{2}\Big(\frac{\p u}{\p x_j}-i\frac{\p u}{\p y_j}\Big)=\frac{1}{2}(u_{x_j}-i u_{y_j}), \qquad u_{\bar j}=\overline{u_j},$$
$$u_{j\bar \ell}=\frac{\p^2 u}{\p z_j \p \bar z_\ell}=\frac{1}{4}\Big\{\big(\frac{\p^2 u}{\p x_j \p x_l}+\frac{\p^2 u}{\p y_j \p y_l}\big)+i\big(\frac{\p^2 u}{\p x_j \p y_l}-\frac{\p^2 u}{\p y_j \p x_l}\big)\Big\}=$$
$$=\frac{1}{4}\Big\{\big(u_{x_j x_l}+ u_{y_j y_l}\big)+i\big(u_{x_j y_l}-u_{y_j x_l}\big)\Big\}$$
Now we define the following $n\times 2n$ matrix
$$J:=\frac{1}{2}\left( I_n, -iI_n\right)$$
and we observe that
$$\partial u =(u_{z_1},\ldots,u_{z_n})=J \; (u_{x_1},\ldots,u_{x_n},u_{y_1},\ldots,u_{y_n})=J\; Du$$
$$\partial \bar \partial u = J \, D^2u \, \bar J^t$$
So, let $\Omega$ be a bounded open set in $\mathbb{C}^n$, we will consider the following Dirichlet problem:
\begin{equation}\label{fullynonlinearliedirichlet}
\left\{
\begin{array}{ll}
    F(z,u,Du,\ddb u)=0, & \hbox{in}\, \Omega, \\
    u=\phi, & \hbox{on} \,\partial \Omega,\\
\end{array}
\right.
\end{equation}
where
\begin{equation}\label{F}
F(z,u,Du,\partial \bar\partial u):=-\big(\sigma_k\big(\partial \bar\partial u\big)\big)^{\frac{1}{k}}+ f(z,u,Du)
\end{equation}
and $\phi$ is a continuous function defined on $\partial \Omega$.\\
We pointed out in the previous definition of $F$ the dependence on the complex Hessian $\ddb u$; anyway we see that this is equivalent to 
$$F_J(z,u,Du,D^2u):=F(z,u,Du,J\,D^2u\, \bar J^t)=-\big(\sigma_k\big(J\,D^2u\, \bar J^t\big)\big)^{\frac{1}{k}}+ f(z,u,Du)$$
\noindent
Since the equation (\ref{sigmak}) is not elliptic in general, we need to give the definition of some suitable cones as in \cite{CNS}. First we define the open cone
$$\Gamma_k=\{\lambda=(\lambda_1, \dots, \lambda_n)\in \mathbb{R}^n: \,\sigma_j( diag(\lambda))
>0,\,\hbox{for every}\, j=1,\ldots,k \},$$
where $diag(\lambda)$ is the $n\times n$ diagonal matrix with entries the $\lambda_j$,
and we denote by $\overline{\Gamma}_k$ and $\partial \Gamma_k$ the closure and the boundary of $\Gamma_k$ respectively.\\

\noindent
We remark that $F$ is elliptic in the cone $\Gamma_k$, i.e.
$$F(z,s,p,M)\leq F(z,s,p,N), \quad \mbox{for all $z\in \mathbb{C}^{n}, s\in \R, p\in \R^{2n},$}$$
and where $M,N$ are $n\times n$ Hermitian matrices whose eigenvalues belong to the open cone $\Gamma_k$ and such that $M\geq N.$\\
We also note that $\big(\sigma_k\big(\cdot)\big)^{\frac{1}{k}}$ is a concave function on the cone $\Gamma_k$. 

Obviously, we also have that $F_J$ is elliptic and concave in the set of $2n\times 2n$ real symmetric  matrices $M$ such that eigenvalues of $J\, M\,\bar J^t$
 belong to the open cone $\Gamma_k.$

Therefore, we give the following
\begin{definition}\label{kconvex}
Let $z_0\in\mathbb{C}^n$ and let $\varphi$ be a $C^2$ function in a neighborhood of $z_0.$ We will say that $\varphi$ is  strictly $k$-plurisubharmonic, in brief strictly $k$-p.s.h.
(respectively $k$-p.s.h.) at $z_0$, if  the vector $\lambda=(\lambda_1, \dots \lambda_n)$ of the eigenvalues of $\partial \bar\partial \varphi(z_0)$  belongs to the open cone $\Gamma_k$ (respectively $\overline{\Gamma}_k$). We remark that the cone  $\Gamma_k$ is invariant with respect to the permutation of $\lambda_j.$\\
We will say that $\varphi$ is  strictly $k$-p.s.h. (respectively $k$-p.s.h.) in $\Omega\subset \C^n$ if $\varphi$ is  strictly $k$-p.s.h. (respectively $k$-p.s.h.) at $z_0$ for every $z_0\in \Omega$.\\
Moreover if $\rho:\mathbb{C}^n \rightarrow \mathbb{R}$ is a smooth defining function for a bounded open set $\Omega\subset \C^n$, that is
$$\Omega=\{z\in\mathbb{C}^{n}:\rho(z)<0\} , \qquad \partial\Omega=\{z\in\mathbb{C}^{n}:\rho(z)=0\}$$
then we will say that the domain $\Omega$ is strictly $k$-p.s.h. if $\rho$ is strictly $k$-p.s.h.
\end{definition}

\noindent
In the sequel we will work essentially in the ball $B_r \subseteq \C^n$: by the previous definition, we see that for any $r>0$ and for any $k=1,\ldots,n$, the ball $B_r$ is strictly $k$-p.s.h., since the defining function $\rho(z)=|z|^2-r^2$ is strictly $k$-p.s.h.\\
We refer to \cite{il}, \cite{cil} for a full detailed exposition on the theory of viscosity solutions: we will give the basic definitions of sub- and super-solution.

\noindent
\begin{definition}\label{soluzioneviscosa}
Let us consider the equation
\begin{equation}\label{equationf}
F(z,u,Du,\ddb u)=0, \qquad \hbox{in}\; \Omega,
\end{equation}
We say that an upper semicontinuous function $u$ (in brief $u\in USC(\Omega)$) is a viscosity sub-solution for (\ref{equationf}) if for every $\varphi\in C^2(\Omega)$, it holds the following: if $z_0 \in\Omega$ is a local maximum for the function $u-\varphi$, then $\varphi$ is $k$-p.s.h. at $z_0$ and
\begin{equation}\label{subsolutionviscosa}
F\big(z_0,u(z_0),D\varphi(z_0),\ddb \varphi(z_0)\big)\leq0.\\
\end{equation}
We say that a lower semicontinuous function $u$ (in brief $u\in LSC(\Omega)$) is a viscosity super-solution for (\ref{equationf}) if for every $\varphi\in C^2(\Omega)$, it holds the following: if $z_0 \in\Omega$ is a local minimum for the function $u-\varphi$, then either
$\varphi$ is $k$-p.s.h. at $z_0$ and
\begin{equation}\label{supersolutionviscosa}
F\big(z_0,u(z_0),D\varphi(z_0),\ddb \varphi(z_0)\big)\geq0\\
\end{equation}
or
$\varphi$ is not $k$-p.s.h. at $z_0.$\\
A continuous function $u$ is a viscosity solution for (\ref{equationf}) if it is either a viscosity sub-solution and a viscosity super-solution for (\ref{equationf}).\\
We say that a function $u\in USC(\overline{\Omega})$ is a viscosity sub-solution for (\ref{fullynonlinearliedirichlet}) if $u$ is a viscosity sub-solution for (\ref{equationf}) and in addition $u\leq\phi$ on $\partial\Omega$.\\
We say that a function $u\in LSC(\overline{\Omega})$ is a viscosity super-solution for (\ref{fullynonlinearliedirichlet}) if $u$ is a viscosity super-solution for (\ref{equationf}) and in addition $u\geq\phi$ on $\partial\Omega$.\\
A viscosity solution for (\ref{fullynonlinearliedirichlet}) is either a viscosity sub-solution and a viscosity super-solution for (\ref{fullynonlinearliedirichlet}).
\end{definition}

\noindent
We will also need a comparison principle for $F$ in the set of $k$-p.s.h. functions to ensure the uniqueness of the viscosity solution of the Dirichlet problem (\ref{fullynonlinearliedirichlet}). By following the analysis on comparison principle in \cite{cil}, we see that if the function $f$ is continuous, positive and  increasing with respect to $u$, then $F$ is proper in the set of $k$-p.s.h. functions, according to the definition in \cite{cil}, and by using the fact that in the set of the Hermitian 
 matrices $M$ such that $\sigma_j(M)\geq 0$ for all $j\leq k,$ the function $-(\sigma_k(M))^{1/k}$  is monotone decreasing, convex and homogeneous of degree one, i.e. 
 $\sigma_k^{1/k}(\lambda M)=\lambda \sigma_k^{1/k}(M)$ for all $\lambda \in \mathbb{R^+}$, then $F$ satisfies the hypotheses in \cite{cil}. So, we have

\begin{proposition}\label{comp}
If the function $f$ is continuous, positive and increasing with respect to $u$, then if
$\underline{u}$ and $\overline{u}$ are respectively viscosity sub-and super-solution of (\ref{fullynonlinearliedirichlet}) in $B_r$, such that $\underline{u}\leq \overline{u}$ on $\partial B_r$, then $\underline{u}\leq\overline{u}$ in $\overline{B}_r$.
\end{proposition}

The paper is organized as follows. In Section \ref{existence}  we  prove the existence of a Lipschitz viscosity solution of (\ref{fullynonlinearliedirichlet}) in $B_r$, as limit of a sequence $\{u^\varepsilon\}$ of smooth solutions of a regularized elliptic problem, whose gradient is bounded independently on $\varepsilon$.
In Section \ref{sec:proofofmaintheorem} we prove our main result:

\begin{theorem}\label{main}
Let us denote by $|\p_z f|, |f_u|, |f_p|$ the derivative of $f$ with respect to its arguments. 
Suppose  $2\leq k\leq n$, and $f\in C^{\infty}( B_1\times \R \times \R^{2n})$ is a  positive function, monotone increasing with respect to $u$ and
 there exists a constant $C>0$ such that $|\p_z f|, |f_u|, |f_p| \leq C.$
Then there exist $R \in (0,1)$ and a $k$-p.s.h.  viscosity solution $u$ to the equation
\begin{equation}\label{LCr}
 F(z,u,Du,\ddb u)= 0 \quad \mbox{in} \quad B_R,
\end{equation}
such that $u \in Lip (\bar B_R).$

Moreover,  if $k>2$ then  $u \not\in C^{1, \beta}(\bar B_R)$ for any $\beta>  1-\frac{2}{k},$  if $k=2$
then $u \not\in C^{1}(\bar B_R).$  
\end{theorem}

The proof of this theorem uses Pogorelov's counterexamples (see \cite[Section 5.5]{G})
 and its extensions developed  in \cite{urbas 1990} , \cite{GLM2013}, \cite{MM}.


\section{Existence of  Lipschitz continuous viscosity solutions}\label{existence}
\noindent
Here we want to prove the existence of a Lipschitz viscosity solution of (\ref{fullynonlinearliedirichlet}) in $B_r$, as limit of a sequence $\{u^\varepsilon\}$ of smooth solutions of a regularized problem, whose gradient is bounded independently on $\varepsilon$.

\noindent
For any given $\varepsilon>0$, we define
$$F^{\varepsilon}(z,u,Du,\ddb u):=-\big(\sigma_k(\ddb u + \varepsilon\; trace(\ddb u) I_n)\big)^{\frac{1}{k}}+f(z,u,Du)$$
It turns out that $-F^{\varepsilon}$ is uniformly elliptic with ellipticity constants depending on $\varepsilon$, in particular we have:
\begin{lemma}\label{uniformlyelliptic}
Let $k=1,\ldots,n$. There exist constants $0<\lambda_{\varepsilon}<\Lambda_{\varepsilon}$, depending on $\varepsilon$, such that
$$\lambda_{\varepsilon}\; trace(N)\leq-F^{\varepsilon}(z,r,p,M+N)+F^{\varepsilon}(z,r,p,M)\leq \Lambda_{\varepsilon}\; trace(N)$$
for every $n\times n$ Hermitian and positive definite matrix $N$, for every $z\in\Omega, r\in {\mathbb R}, p\in {\mathbb R}^{2n} $, and for every $n\times n$ Hermitian matrix $M$ such that
the eigenvalues of  $M+\varepsilon \;trace (M)\; I_n$ are in the cone $\Gamma_k.$
\end{lemma}

\begin{proof}
It's a straightforward computation, by taking into account that the functions $\big(\sigma_k\big)^{\frac{1}{k}}$ are homogeneous of degree one, monotone increasing and concave. We have
\[
\begin{split}
&-F^{\varepsilon}(z,r,p,M+N)+F^{\varepsilon}(z,r,p,M)\\
&=\Big(\sigma_k\left( M+N+\varepsilon \, trace(M+N)I_n\right)\Big)^{\frac{1}{k}}-
\Big(\sigma_k\left( M +\varepsilon \,trace(M)I_n\right)\Big)^{\frac{1}{k}}\\
&\geq \Big(\sigma_k\left(N +\varepsilon  \, trace(N)I_n\right)\Big)^{\frac{1}{k}}\geq
\Big(\sigma_k\left( \varepsilon \, trace(N)I_n\right)\Big)^{\frac{1}{k}}\\
&\geq   \varepsilon \, trace( N ) \Big(\sigma_k \left(I_n\right)\Big)^{\frac{1}{k}}
= \lambda_{\varepsilon} \, trace( N )
\end{split}
\]
Moreover, by the monotonicity, the homogeneity and the concavity
of  $\big(\sigma_k\big)^{\frac{1}{k}}$  and by Lagrange theorem there is $\theta \in ]0,1[$ such that
\[
\begin{split}
&-F^{\varepsilon}(z,r,p,M+N)+F^{\varepsilon}(z,r,p,M)\\
&=\Big(\sigma_k\left( M+N+\varepsilon \, trace( M+N)I_n\right)\Big)^{\frac{1}{k}}-
\Big(\sigma_k\left( M +\varepsilon \,trace(M)I_n\right)\Big)^{\frac{1}{k}}\\
&\leq\Big(\sigma_k\left( M+\varepsilon \, trace( M)I_n + (1+\varepsilon) \, trace(N)I_n\right)\Big)^{\frac{1}{k}}-
\Big(\sigma_k\left( M +\varepsilon \,trace(M)I_n\right)\Big)^{\frac{1}{k}}\\
&=\Big(\partial_{r_{jj}}\big(\sigma_k\big)^{\frac{1}{k}}\left(M+\varepsilon \, trace( M)I_n + \theta(1+\varepsilon) \, trace(N)I_n\right)\Big) \; \cdot  (1+\varepsilon) \, trace(N)\\
&\leq \Big(\partial_{r_{jj}} \big(\sigma_k\big)^{\frac{1}{k}}\left(\theta(1+\varepsilon)\, trace(N)I_n\right)\Big) \cdot  (1+\varepsilon) \,trace(N)\\
&= \Big(\partial_{r_{jj}} \big(\sigma_k\big)^{\frac{1}{k}}\left(I_n\right)\Big) \cdot  (1+\varepsilon) \, trace(N)=\Lambda_{\varepsilon} \, trace\left(   N \right) \\
\end{split}
\]
where $ \partial_{r_{jj}}\left( \big(\sigma_k(r)\big)^{\frac{1}{k}}\right)$ denotes the sum in $j$ of  partial derivatives of  $\big(\sigma_k\big)^{\frac{1}{k}}$ with respect to $r_{jj}$, which is homogeneous of degree zero.
\end{proof}

\begin{remark}
Since
$$trace(\ddb u)= trace(J\; D^2 u\; \bar J^t)=\frac{1}{4}trace(D^2 u)=\frac{1}{4}\Delta u$$
the previous lemma states the uniform ellipticity for any fixed $\varepsilon>0,$ with respect to the classical Euclidian Hessian, of the functional
$$F^{\varepsilon}_J(z,u,Du,D^2u):=-\big(\sigma_k(J\;D^2 u\; \bar J^t + \frac{\varepsilon}{4}\;  \Delta u\;I_n)\big)^{\frac{1}{k}}+f(z,u,Du)$$
in the set of the $2n \times 2n$ symmetric matrices $M$ such that the eigenvalues of
$$J\;M\; \bar J^t + \frac{\varepsilon}{4}\;trace (M)\; I_n$$
are in the cone $\Gamma_k.$
\end{remark}
\noindent
We will need some structural assumptions on $f$, in order to ensure the existence of a smooth solution.  Indeed, we assume that $f$ is a smooth function and we require on $B_1 \times \R \times \R^{2n}$ the following hypotheses:
\begin{itemize}
  \item[(H1)] $f_u\geq0$;
  \item[(H2)] there exists a constant $C$ such that $|\p_z f|, |f_u|, |f_p| \leq C.$
\end{itemize}
We notice that hypothesis $(H1)$ will give us the uniqueness of the solution, by the comparison principle; on the other hand $(H2)$ will imply bounds for the second derivatives and for their H\"older seminorms as in \cite{trudi}, \cite{kry}, since $F^{\varepsilon}$ is uniformly elliptic in the sense of \cite{trudi}.  Estimates for higher derivatives follow from the linear uniformly elliptic theory \cite[Lemma 17.16]{gilbarg-trudinger}.
These estimates allow us to apply the method of continuity \cite[Theorem 17.8]{gilbarg-trudinger}.
Therefore, under hypotheses $(H1), (H2)$, for every $\phi \in C^{2,\alpha}(\bar B_r)$, $r\leq 1$, there exists a (unique) classical solution $u^{\varepsilon}\in C^{2,\alpha}(B_r)$ of the Dirichlet problem related to $F^{\varepsilon}=0$ in $B_r$ (from further regularity results, $u^{\varepsilon}$ is actually $C^\infty$); moreover  $u^{\varepsilon}$ is strictly $k$-p.s.h.\\
We will prove a gradient bound for $u^{\varepsilon}$, uniform in $\varepsilon$. Thus, by taking the uniform limit as $\varepsilon$ goes to zero we will find a Lipschitz continuous viscosity solution $u$ of the Dirichlet problem related to $F=0$; in this last limit process we can use the stability property of the viscosity solutions with respect to the uniform convergence, since the sets of $k$-p.s.h. functions satisfy a crucial property of inclusion as $\varepsilon$ decreases (see for instance \cite{urbas 1990}).\\

\noindent
We will make use of particular sub- and super-solutions that we will build with the help of a suitable convex function $\phi$.

\noindent
We have, for $2\leq k \leq n$:
\begin{proposition} \label{gradientestimates1}
Let $f$ be a positive function satisfying $(H1), (H2)$ and let $\phi \in C^{2,\alpha} (\overline{B}_1)\cap Lip (\overline{B}_1)$ be a convex function such that $\sigma_k(\ddb \phi)=0$ in $B_1$. Then there exists $r_0\leq 1$ such that for any $0<r< r_0$, the problem (\ref{fullynonlinearliedirichlet}) in $B_r$ has a viscosity solution $ u \in \mbox{Lip}\, (\overline{B}_r)$ satisfying
\begin{equation}\label{aprioriestimate}
\|u \|_{L^{\infty}( \overline{B}_r)}\, + \, \| u \|_{Lip\, (\overline{B}_r)}\,\, \leq \,\, C,
\end{equation}
where $C$ only depends on $r$, $\|\phi \|_{L^{\infty} (\overline{B}_1)}$, $ \| D \phi \|_{L^{\infty} (\overline{B}_1)}$
\end{proposition}

\noindent
We will prove the Proposition \ref{gradientestimates1} in some steps. First of all, since $-F^{\varepsilon}$ is uniformly elliptic, by the hypotheses on $f$, there exists a smooth solution $u^{\varepsilon}$ of the problem
\begin{equation}\label{fullynonlinearliedirichletvarepsilondelta}
\left\{
\begin{array}{ll}
    F^{\varepsilon}(z,u,D u,\ddb u)=0, & \hbox{in}\, B_r, \\
    u=\phi, & \hbox{on} \,\partial B_r,\\
\end{array}
\right.
\end{equation}
Moreover, applying the technique of \cite[Lemma 1.4]{CKNS}, we get
\begin{equation}\label{stimagradienteinterno}
\sup_{B_r}|D u^{\varepsilon}| \leq  \sup_{\partial B_r} |D u^{\varepsilon}|+c
\end{equation}
where $c$ is a positive constant independent of $\varepsilon$ and only depending on the constant $C$ in $(H2)$.
\noindent
Now, we are going to find explicitly global sub- and super-solutions.
\noindent
\begin{lemma}\label{lem1}
Let $\phi \in C^2 (\overline{B}_1)\cap Lip (\overline{B}_1)$ be a convex function such that $\sigma_k(\ddb \phi)=0$ in $B_1$. Then there exists $\varepsilon_0\leq 1$ such that for every $0<\varepsilon< \varepsilon_0$ we have
\begin{equation}\label{subs1}
F^{\varepsilon}(z, \phi, D \phi, \ddb \phi) >0 
\end{equation}
\end{lemma}
\begin{proof}
Let us define
$$\inf_{B_1\times\R\times\R^{2n}} f =c_0>0$$
Therefore
$$F^{\varepsilon}(z, \phi, D \phi, \ddb \phi)=-\sigma_k\left(\partial \bar \partial \phi +\varepsilon \;trace(\ddb \phi) I_n \right)^{1/k}+f(z,\phi, D \phi)\geq $$
$$\geq -\sigma_k\left(\partial \bar \partial \phi +\varepsilon\;trace(\ddb \phi) I_n \right)^{1/k}+c_0$$
We know there exist positive constants $\{c_1, \ldots,c_n\}$ such that
$$0\leq c_n \left(\sigma_n\left(\partial \bar \partial \phi \right)\right)^{1/n}\leq \ldots \leq c_k \left(\sigma_k\left(\partial \bar \partial \phi \right)\right)^{1/k}\leq \ldots\leq c_1 \;trace(\ddb \phi)$$
Since $\sigma_k(\ddb \phi)=0$ in $B_1$, we have that, if
$$\sup_{z\in B_1} \;trace(\ddb \phi(z))=0$$
then $\ddb \phi =0$ in $B_1$, therefore for any positive $\varepsilon$
$$F^{\varepsilon}(z, \phi, D \phi, \ddb \phi)= f(z,\phi,\p \phi)\geq c_0>0$$
in $B_1$. On the other hand, let
$$\sup_{z\in B_1} \;trace(\ddb \phi(z))=M>0$$
then we have, for $0<\varepsilon\leq1$,
$$\sigma_k\left(\partial \bar \partial \phi +\varepsilon \;trace(\ddb \phi) I_n \right)=\sigma_k\left(\partial \bar \partial \phi\right)+(\varepsilon \;trace(\ddb \phi))\sigma_{k-1}\left(\partial \bar \partial \phi\right)+\ldots$$
$$+(\varepsilon \;trace(\ddb \phi))^{k-1}\;trace(\ddb \phi)+(\varepsilon \;trace(\ddb \phi))^k\leq$$
$$\leq (\varepsilon \;trace(\ddb \phi))\frac{c_1}{c_{k-1}}(\;trace(\ddb \phi))^{k-1}+ \ldots$$
$$+ (\varepsilon \;trace(\ddb \phi))^{k-1}\;trace(\ddb \phi)+(\varepsilon \;trace(\ddb \phi))^k \leq$$
$$\leq k c^* \varepsilon (\;trace(\ddb \phi))^k \leq k c^* \varepsilon M^k $$
where
$$c^*=\max\Big\{\frac{c_1}{c_{k-1}}, \ldots, \frac{c_1}{c_{2}} \Big\}$$
Therefore, if we let $\varepsilon_0=\min\{1,\frac{c_0^k}{k c^* M^k}\}$, we get
$$F^{\varepsilon}(z, \phi, D \phi, \ddb \phi)=-\sigma_k\left(\partial \bar \partial \phi +\varepsilon \;trace(\ddb \phi) I_n \right)^{1/k}+f(z,\phi,D \phi)\geq$$
$$\geq - (k c^* \varepsilon M^k)^{1/k}+c_0>0$$
for any $0<\varepsilon< \varepsilon_0$.
\end{proof}

\begin{lemma}\label{main2}
Let $\phi \in C^2 (\overline{B}_1)\cap Lip (\overline{B}_1)$ be a convex function. For positive $\lambda$ and $r$, we define
$$  u_{\lambda} (z) : = \phi (z)+ \lambda \rho(z),  $$
where $\rho(z)   =|z|^2-r^2$. Then, there exists $\lambda^* \geq1$, depending on $\sup_{B_1} |\phi|$ and $\sup_{B_1} |D\phi|$, such that
\begin{equation}\label{subs2}
F^\varepsilon (z, u_{\lambda^*}, D u_{\lambda^*}, \ddb u_{\lambda^*}) <0 \,\, \mbox{in}\,\, B_r,
\end{equation}
for any $r<1/\lambda^*$.
\end{lemma}
\begin{proof}
Since $\phi$ is a plurisubharmonic function,
$$\partial \bar \partial u_{\lambda}= \partial \bar \partial \phi +\lambda I_n\geq \lambda I_n$$
In particular $u_{\la}$ is $k$-p.s.h. for every $\lambda >0$ and
$$\ddb u_\la + \varepsilon \;trace (\ddb u_\la)\;I_n \geq \ddb u_\la$$
Moreover, as a consequence of the monotonicity of the function $\sigma_k^{1/k}(\cdot)$ and of its homogeneity, i.e.  $\sigma_k^{1/k}({\la}\cdot)= {\la}  \sigma_k^{1/k}(\cdot)$ for every $\la>0,$   we get
\[
 \begin{split}
 & F^\varepsilon(z,u_{\la}, D u_{\la},  \ddb u_{\la}) \leq F(z,u_{\la}, D u_{\la},  \ddb u_{\la}) =\\
  &= -\sigma_k\left(\partial \bar \partial u_{\lambda} \right)^{1/k}+f(z,u_{\la}, D u_{\la})\leq -{\la} + f(z, \phi (z)+ \lambda \rho(z), D \phi (z)+ 2\lambda z)
\end{split}
 \]
Now, let
$$L:=\sup_{z \in B_1} |D \phi(z)|, \qquad E:=\sup_{(z,p)\in B_1 \times B_{L+2}} f(z,\sup_{B_1} \phi, p), \qquad \lambda^*:=\max\{1,E\}$$
so, if $r<1/\lambda^*$ and $z\in B_r$, it holds
$$|D u_{\lambda^*} (z)|=|D \phi (z)+ 2\lambda^* z | \leq \sup_{B_r} |D \phi|+ 2\lambda^* r \leq \sup_{B_1} |D\phi|+ 2\lambda^* r \leq L+2$$
Then, since $\rho$ is negative in $B_r$ and $f$ is increasing with respect to $u$, we have for any $r<1/\lambda^*$
$$f(z, \phi (z)+ \lambda^* \rho(z), D \phi (z)+ 2\lambda^* z)\leq f(z, \sup_{B_1} \phi, D \phi (z)+ 2\lambda^* z)\leq $$
$$\leq \sup_{(z,p)\in B_1 \times B_{L+2}} f(z,\sup_{B_1} \phi, p)\leq \lambda^*$$
and this ends the proof.
\end{proof}

\begin{proof}[Proof of Proposition \ref{gradientestimates1} ]
Let $u_{\lambda^*} = \phi + \lambda^* \rho$ be the function given by the $\lambda^*$ as defined in the previous lemma. So, if $r<r_0:=1/\lambda^*$, then $u_{\lambda^*} \in C^2(\overline{B}_r)$ and it is a classical sub-solution to  $F^\varepsilon=0$ in $B_r$. Moreover,  $u_{\lambda^*} = \phi$ on $\partial B_r$.\\
On the other hand, if $\varepsilon<\varepsilon_0$ then $\phi$ is a classical  super-solution  to  $F^\varepsilon=0$ in $B_r$.\\
So, by the comparison principle we have
$$ u_{\lambda^*} \leq u^\varepsilon \leq \phi \mbox{ in } \overline{B}_r, \quad r<r_0, \quad \varepsilon<\varepsilon_0$$
Since
$$|u_{\lambda^*} (z)|= |\phi (z)+ \lambda^* \rho(z) |\leq \sup_{B_r} |\phi| + \lambda^* r^2\leq \|\phi \|_{L^{\infty} (\overline{B}_1)}+r$$
and
$$|D u_{\lambda^*} (z)|=|D \phi (z)+ 2\lambda^* z | \leq \sup_{B_r} |D \phi|+ 2\lambda^* r \leq \|D \phi \|_{L^{\infty} (\overline{B}_1)}+2$$
then, if $u$ denotes the uniform limit of $u^\varepsilon$ as $\varepsilon$ goes to zero, we  can conclude that $u \in \mbox{Lip}\, (\overline{B}_r)$ independently on $\varepsilon$ and
\begin{equation}
\|u \|_{L^{\infty}( \overline{B}_r)}\, + \, \| u \|_{Lip\, (\overline{B}_r)}\,\, \leq \,\, C,
\end{equation}
with $C$ depending on $r$, $\|\phi \|_{L^{\infty} (\overline{B}_1)}$, $ \| D \phi \|_{L^{\infty} (\overline{B}_1)}$
\end{proof}

\section{Existence of nonsmooth solutions}\label{sec:proofofmaintheorem}

\noindent
Throughout this section we denote by
$$ z=(z_1, z',z''),\,\, z' = (z_2, \dots, z_k), \quad z''=(z_{k+1}, \dots, z_n)$$
For  $ 0 \leq \varepsilon <1$ and $ 0 < r < r_0$ such that Proposition \ref{gradientestimates1} holds true, we define
\begin{equation}\label{w2bis}
 w_{\varepsilon} (z) = w_{\varepsilon}(z) : = (r^2+|z_1|^2)(\varepsilon^2+ |z'|^2)^{\alpha}, \quad \alpha = 1- \frac{1}{k},
\end{equation}
and
$$
\psi_{\varepsilon}(z) : = M w_{\varepsilon}(z), \quad \phi_{\varepsilon} (z) = \phi_{\varepsilon}(z) : = 2 M ( \varepsilon^2 + |z'|^2)^{\alpha},
$$
with $M$ a positive constant to be determined.

\begin{lemma}\label{lemma1}
There exists $M = M(r)$ such that
\begin{equation}\label{claim}
F(z, \psi_{\varepsilon}, D \psi_{\varepsilon}, \partial \bar \partial\psi_{\varepsilon} ) < 0 \quad \mbox{in}\,\, B_r, \quad \forall\,\varepsilon \in ]0, r[.
\end{equation}
\end{lemma}
\begin{proof}
First, $w_\varepsilon (z)$ is independent on $z''$ therefore $\partial \bar\partial w_\varepsilon (z)$ has $n-k$ null rows by construction and so
\begin{equation}\label{det}
\s_k(\partial \bar\partial w_\varepsilon (z))=\det \partial \bar\partial_{(z_1,z')} w_\varepsilon (z).
\end{equation}
Direct computations show that
\begin{equation}\label{eq:formulafordetofwsubsigma}
\det \partial \bar\partial_{(z_1,z')} w_\varepsilon (z)= f_\varepsilon(z)
\end{equation}
with
 \[
 f_\varepsilon(z)=
 \alpha^{k} (r^2+|z_1|^2)^{k-2}\; \frac{r^2(\alpha^{-1}\varepsilon^2 +|z'|^2)+\alpha^{-1}\varepsilon^2 |z'|^2}{(\varepsilon^2+ |z'|^2)}.
\]
Indeed, we have
\[
\begin{split}
\partial \bar\partial_{(z_1,z')} w_\varepsilon (z)
=&(\varepsilon^2+|z'|^2)^{\alpha-1}
\left(%
\begin{array}{ll}
\varepsilon^2+|z'|^2
&
\alpha \bar z_1 \; z'\\
\alpha z_1 \;(\bar z')^T
&
(r^2+|z_1|^2)\left(\alpha I_{k-1}+\alpha
(\alpha-1)\frac{z'\otimes \bar z'}{{\varepsilon^2+|z'|^2}}\right)
\end{array}%
\right)\\
\end{split}
\]
Since $(\alpha-1)k=-1$, we get
\[
\begin{split}
&\det   \partial \bar\partial_{(z_1,z')} w_\varepsilon (z)=\\
=&(\varepsilon^2+|z'|^2)^{-1}\det
\left(%
\begin{array}{ll}
\varepsilon^2+|z'|^2
&
\alpha \bar z_1 \; z'\\
\alpha z_1 \;(\bar z')^T
&
(r^2+|z_1|^2)\left(\alpha I_{k-1}+\alpha
(\alpha-1)\frac{z'\otimes \bar z'}{{\varepsilon^2+|z'|^2}}\right)
\end{array}%
\right)\\
=&\det
\left(%
\begin{array}{ll}
1
&
\alpha (\varepsilon^2+|x'|^2)^{-1/2} \; \bar z_1 \; z'\\
\alpha (\varepsilon^2+|x'|^2)^{-1/2} \; z_1 \;(\bar z')^T
&
(r^2+|z_1|^2)\left(\alpha I_{k-1}+\alpha
(\alpha-1)\frac{z'\otimes \bar z'}{{\varepsilon^2+|z'|^2}}\right)
\end{array}%
\right)\\
=&\det
\left(%
\begin{array}{ll}
1
&
0\\
\alpha (\varepsilon^2+|x'|^2)^{-1/2} \; z_1 \;(\bar z')^T
&
(r^2+|z_1|^2)\left(\alpha I_{k-1}+\alpha
(\alpha-1)\frac{z'\otimes \bar z'}{{\varepsilon^2+|z'|^2}}\right)-\alpha^2 |z_1|^2 \frac{z'\otimes \bar z'}{{\varepsilon^2+|z'|^2}}
\end{array}%
\right)\\
=&\det \left(%
\begin{array}{l}
(r^2+|z_1|^2)\left(\alpha I_{k-1}+\alpha
(\alpha-1)\frac{z'\otimes \bar z'}{{\varepsilon^2+|z'|^2}}\right)-\alpha^2 |z_1|^2 \frac{z'\otimes \bar z'}{{\varepsilon^2+|z'|^2}}
\end{array}%
\right)\\
:=&\det \Gamma,\\
\end{split}
\]
where $\Gamma$ is a $(k-1) \times (k-1)$ symmetric matrix. It is easy to see that $\lambda_1=\alpha (r^2+|z_1|^2)$ is an eigenvalue of $\Gamma$ with multiplicity $k-2.$ Now, $\trace \, \Gamma
=(k-2)\lambda_1+\lambda_2$ with
\[
\begin{split}
\lambda_2=&(r^2+|z_1|^2)\left(\alpha+\alpha
(\alpha-1)\frac{|z'|^2}{{\varepsilon^2+|z'|^2}}\right)-\alpha^2|z_1|^2\frac{|z'|^2}{{\varepsilon^2+|z'|^2}}\\
=&\alpha^2\,
\frac{r^2\left(\frac{\varepsilon^2}{\alpha}+|z'|^2\right)+\frac{\varepsilon^2}{\alpha}|z'|^2}{{\varepsilon^2+|z'|^2}}
\end{split}
\]
Thus
$$\det \Gamma =\lambda_1^{k-2}\lambda_2=\alpha^{k} (r^2+|z_1|^2)^{k-2}\frac{r^2\left(\frac{\varepsilon^2}{\alpha}+|z'|^2\right)+
\frac{\varepsilon^2}{\alpha}|z'|^2}{{\varepsilon^2+|z'|^2}}=f_\varepsilon$$
which completes the proof of \eqref{eq:formulafordetofwsubsigma}.
In particular, if $\varepsilon < r$
$$f_\varepsilon\geq {\alpha^{k} r^{2(k-1)}}\geq \frac{\alpha^{k-1} r^{2 (k-1)}}{2}$$
Since $\psi_{\varepsilon} =  M w_{\varepsilon}$, by \eqref{det}, we can choose $r$ small, such that
$$
\sigma_k\left( \ddb \psi_{\varepsilon}\right)^{1/k}> \alpha^\alpha  r^{2\alpha}M.
$$
On the other side, direct computations show that
\[
 \begin{split}
 |D w_{\varepsilon}|^2 =
 4\Big( &|z_1|^2( \varepsilon^2 + |z'|^2)^{2 \alpha}\,
 + \, \alpha^2 |z'|^2 ( r^2 + |z_1|^2)^2
 ( \varepsilon^2 +|z'|^2)^{2 (\alpha - 1)}\Big)
 \end{split}
\]
and for every $\varepsilon \in ]0, r[$,
\begin{equation}\label{boundgradient}
|D w_{\varepsilon}|^2 \leq 2^{2 \alpha + 3} r^{4 \alpha + 2}  \quad
\mbox{in} \,\, B_r.
\end{equation}
From  \eqref{boundgradient}, we obtain
\begin{equation}\label{loverbound}
|D \psi_{\varepsilon}| \leq 2^{ \alpha + 3/2} M r^{2 \alpha + 1}  \quad
\mbox{in} \,\, B_r.
\end{equation}
Choosing $ M = 2^{- \alpha-(3/2)} r^{-2 \alpha -1}$,
the right hand side of (\ref{loverbound}) equals $1,$
and $ \psi_{\varepsilon}\leq 2^{-1/2}r<1.$
The strategy now is to take $r$ such that
\[
\sup_{(z,p)\in B_1\times B_1} f(z,1, p)<  \alpha^\alpha 2^{-\alpha-(3/2)} r^{ -1}.
\]
Then, by the increasing monotonicity of $s\rightarrow  f(\cdot, s, \cdot),$    in $B_r$ we obtain
\[
 \begin{split}
 F(z, \psi_{\varepsilon} , D \psi_{\varepsilon}, \ddb \psi_{\varepsilon}) &  <-\alpha^{\alpha}2^{-\alpha-(3/2)} r^{ -1}+
 f(z, \psi_{\varepsilon} , D \psi_{\varepsilon})\\
 &<
 -\alpha^{\alpha}2^{-\alpha-(3/2)} r^{ -1}+ f(z,1, D \psi_{\varepsilon})  <0.
\end{split}
\]
\end{proof}

\begin{proof}[Proof of Theorem \ref{main}]
We have
$$
 \psi_0 \leq \psi_{\varepsilon} \leq \phi_{\varepsilon} , \quad \mbox{in}\,\, B_1.
$$
Since $ k\geq2$, the exponent $\alpha\geq\frac{1}{2}$ and so $\phi_{\varepsilon}$ is convex for $\varepsilon\geq 0.$

Moreover, $\phi_{\varepsilon}$ is smooth for $\varepsilon>0$,
and independent of $z_1$ and $z''$, therefore $\partial \bar \partial \phi_\varepsilon$ has $n-k+1$ null eigenvalues.
Therefore:
\begin{equation}\label{supersolution}
F(z, \phi_{\varepsilon}, D \phi_{\varepsilon}, \partial \bar \partial\phi_{\varepsilon}) = f(z, \phi_{\varepsilon}, D \phi_{\varepsilon}) > 0 \quad\mbox{in}\,\, B_1, \quad \forall \,\varepsilon \in ]0,1[.
\end{equation}
Thus, applying Proposition \ref{gradientestimates1}, there exists $ 0 < r< r_0$ such that the Dirichlet problem
$$
F= 0 \quad \mbox{in}\,\, B_r, \qquad u = \phi_{\varepsilon}\quad \mbox{on}\,\, \partial B_r,
$$
with $\varepsilon \in ]0,1[$, has a viscosity solution $u_{\varepsilon}$ such that
$$
  \| u_{\varepsilon} \parallel_{L^{\infty}( \overline{B}_r)}\, + \, \parallel u_{\varepsilon} \parallel_{Lip\, (\overline{B}_r)}\,\, \leq \,\, C(r,\varepsilon, M)
$$
with $ C(r,\varepsilon, M)$ depending on $\varepsilon$ only through
$$ C(\phi_{\varepsilon}):= \parallel \phi_{\varepsilon} \parallel_{L^{\infty}( \overline{B}_r)}\, + \, \parallel D \phi_{\varepsilon} \parallel_{L^{\infty} (\overline{B}_r)}.$$
On the other hand, an elementary computation shows that $C(\phi_{\varepsilon}) \leq 8 M $. Then, we can choose $C(r,\varepsilon, M)$ independent of $\varepsilon$, and so
\begin{equation}\label{uniformestimate}
 \parallel u_{\varepsilon} \parallel_{L^{\infty}( \overline{B}_r)}\, + \, \parallel u_{\varepsilon} \parallel_{Lip\, (\overline{B}_r)}\,\, \leq \,\, C(r, M ).
\end{equation}
Now we can use the Comparison Principle to compare $u_{\varepsilon}$
with $\psi_{\varepsilon}$ and $ \phi_{\varepsilon}$. Indeed, if $\varepsilon<r$, by
(\ref{supersolution}) and Lemma \ref{lemma1}, $\phi_{\varepsilon}$ and
$\psi_{\varepsilon}$ are, respectively, classical super-solution and
sub-solution to $F= 0$ in $B_r$. On the other hand
$\psi_{\varepsilon} \leq \phi_{\varepsilon}$ in $B_1$, in
particular, $\psi_{\varepsilon} \leq \phi_{\varepsilon}$ on $\partial B_r$. Thus, by the Comparison Principle,
\begin{equation}\label{comparison}
 \psi_{\varepsilon} \leq u_{\varepsilon} \leq \phi_{\varepsilon} \quad \mbox{in}\,\, B_r, \,\,\, \forall \varepsilon \in ]0, r[.
\end{equation}
The uniform estimate (\ref{uniformestimate}) implies the existence of a sequence $ \varepsilon_j \searrow 0$ such that $(u_{\varepsilon_j})_{j \in \N}$ uniformly converges to a viscosity solution $ u \in Lip(\overline{B}_r)$ to the Dirichlet problem
$$
F= 0 \quad \mbox{in}\,\, B_r, \quad u = \phi_{0}\,\,\mbox{on}\,\,\partial B_r.
$$
Moreover, from the comparison principle, we get
\begin{equation}\label{key0}
 \psi_0 \leq u \leq \phi_0 \quad \mbox{in}\,\, B_r.
\end{equation}
So
\begin{equation}\label{key}
 M r^2 |z_1|^{2\alpha} \leq u(z_1,0,\dots,0) \leq 2 M |z_1|^{2\alpha}.
\end{equation}
and in particular
\begin{equation}\label{keyx}
 M r^2 |x_1|^{2\alpha} \leq u((x_1,0),0,\dots,0) \leq 2 M |x_1|^{2\alpha}.
\end{equation}
As in the proof of \cite[Theorem 1]{GLM2013}
inequalities   in \eqref{keyx} imply, if $k>2$:
$$\p_{x_1} u \notin C^{\beta}, \,\, \mbox{for every}\,\,\beta >  2 \alpha - 1 =1-\frac{2}{k}\quad \mbox{if}\,\,  2 \alpha > 1\,\,\,(\mbox{i.e.}\,\, k >2).
$$
Indeed, if $ 2\alpha > 1$, then $\partial_{x_1} u(0 )= 0 = u (0)$ so that, if $ u$ was $C^{1,\beta}$, with $\beta>2\alpha-1$, we would have $ u ((x_1,0),0,\dots,0) \leq C |x_1|^{1+\beta}$
for a suitable $C>0$ and for every $x_1$ sufficiently small. Hence, by the first inequality in (\ref{keyx}), we would have $\beta \leq 2 \alpha - 1 $, a contradiction.\\
If $k=2$, in the same way, we see that $\p_{x_1} u$ is not continuous 
and this ends the proof.
\end{proof}

\section*{Acknowledgements}

Vittorio Martino and Annamaria Montanari are members of the {\it Gruppo Nazionale per
l'Analisi Matematica, la Probabilit\`a e le loro Applicazioni} (GNAMPA)
of the {\it Istituto Nazionale di Alta Matematica} (INdAM).

\end{document}